\documentclass[10pt]{article}

\usepackage{graphicx}
\usepackage{geometry,verbatim} 
\usepackage{bm}
\usepackage{subfig}
\usepackage{tikz}
\usetikzlibrary{arrows,petri,topaths,positioning,decorations.markings}
\usepackage{anysize}               		
\usepackage{amsmath,amsbsy,amsfonts,graphics,epsfig,float,color}
\usepackage{amssymb}
\usepackage{array}

\newcommand\floor[1]{\lfloor#1\rfloor}
\newtheorem{theorem}{Theorem}[section]
\newtheorem{lemma}[theorem]{Lemma}
\newtheorem{proposition}[theorem]{Proposition}
\newtheorem{corollary}[theorem]{Corollary}

\newenvironment{proof}[1][Proof]{\begin{trivlist}
\item[\hskip \labelsep {\bfseries #1}]}{\end{trivlist}}
\newenvironment{definition}[1][Definition]{\begin{trivlist}
\item[\hskip \labelsep {\bfseries #1}]}{\end{trivlist}}

\newcommand{\qed}{\nobreak \ifvmode \relax \else
      \ifdim\lastskip<1.5em \hskip-\lastskip
      \hskip1.5em plus0em minus0.5em \fi \nobreak
      \vrule height0.75em width0.5em depth0.25em\fi}

\begin{document}
\date{}
\title{An Unoriented Variation on de Bruijn Sequences \thanks{This work was motivated by results in \cite{FP,FBP}, supported by NSF grant DMS-1022635 to PDS.
}
}

\author{Christie S. Burris\thanks{Department of Mathematics,
              Virginia Tech, csburris@vt.edu}         \and
        Francis C. Motta\thanks{Department of Mathematics,
              Duke University, motta@math.duke.edu} \and
        Patrick D. Shipman\thanks{Department of Mathematics,
             Colorado State University, shipman@math.colostate.edu}}

\maketitle

\begin{abstract}
For positive integers $k,n$, a de Bruijn sequence $B(k,n)$ is a finite sequence of elements drawn from $k$ characters whose subwords of length $n$ are exactly the $k^n$ words of length $n$ on $k$ characters. This paper introduces the unoriented de Bruijn sequence $uB(k,n)$, an analog to de Bruijn sequences, but for which the sequence is read both forwards and backwards to determine the set of subwords of length $n$. We show that nontrivial unoriented de Bruijn sequences of optimal length exist if and only if $k$ is two or odd and $n$ is less than or equal to 3. Unoriented de Bruijn sequences for any $k$, $n$ may be constructed from certain Eulerian paths in Eulerizations of unoriented de Bruijn graphs.
\end{abstract}

%%%%%%%%%%%%%%%%%%%%%

\section{Unoriented de Bruijn sequences}

For positive integers $k$ and $n$, what is the minimal length of a word over an alphabet of size $k$ which contains every length-$n$ word as a subword? The minimum possible length of such a word is $k^n+n-1$, as this length is required to see all $k^n$ such words without repetition.  What is less clear is that for each $k, n$ there are many words which achieve this lower bound. Such a word is called a de Bruijn sequence $B(k,n)$; its $k^n$ subwords of length $n$ are exactly the set of  $k^n$ words of length $n$ on $k$ characters.  An example of a de Bruijn sequence for $k=2$, $n=3$ is 0100011101 since it contains every binary word of length 3 -- 010, 100, 000, 001, 011, 111, 110, 101 -- exactly once. A de Bruijn sequence, $B(k,n)$, corresponds to an Eulerian circuit in a so-called de Bruijn graph, $Bg(k,n)$, whose $k^{n}$ directed edges are labeled by the length-$n$ $k$-ary words \cite{deB,deB2}. 

In this paper, we introduce a variation on the idea of a de Bruijn sequence, as exemplified by the sequence 00010111.  This sequence has each of the binary palindromes of length 3 as subwords, namely 000, 010, 101, and 111, but only one member of each of the pairs \{001, 100\} and \{011, 110\}.  Each non-palindrome of length 3 appears either forwards or backwards exactly once and each palindrome appears exactly once when the sequence is read forwards.  We call such a sequence an \textit{unoriented de Bruijn sequence of optimal length}.

 We refer to two words  $\bm{v}$ and $\bm{v'}$ from an alphabet of size $k$ as \textit{reflections} or \textit{reflected pairs} if $\bm{v}=v_1 v_2 \cdots v_{n-1} v_n$ and $\bm{v'}= v_n v_{n-1}\cdots v_2 v_1$.  If the $k$ symbols are $0,1,2,....k-1$, we denote a pair of reflections by $[\bm{v}]$, where $\bm{v}$ is the larger of the two integers written in a $k$-ary expansion.
 
%%%%%%%%%%%%%%%%
\begin{definition}
An \textit{unoriented de Bruijn sequence} $uB(k,n)$ is a sequence of characters drawn from an alphabet $\Sigma_k$ of $k$ symbols that (i) contains as subwords a member from each of the length-$n$ reflections on $k$ symbols, and (ii) is of the shortest length amongst all such sequences satisfying property (i).
\end{definition} 
%%%%%%%%%%%%%%%%

An unoriented de Bruijn sequence exists for any choice of $k$ and $n$ since there always exists a word which contains as subwords a member from each of the length-$n$ reflections on $k$ symbols. Thus there must be a word of minimal length which satisfies this requirement on subwords. The question becomes whether or not this minimal length is optimal, as it is for de Bruijn sequences. If so, then such an unoriented de Bruijn sequence would (unavoidably) see every $k$-ary palindrome of length $n$ twice, but every $k$-ary non-palindrome of length $n$ exactly once when read forwards and backwards. Equivalently, when read forward, such a word would contain as length-$n$ subwords exactly one member of each reflected pair. Thus the minimal possible length of any $uB(k,n)$ is equal to the sum of the number of reflections of length $n$ on $k$ symbols plus $n-1$.  To the number of palindromes, $k^{\lceil n/2 \rceil}$, one adds half of the number of non-palindromes, $(k^n - k^{\lceil n/2 \rceil})/2$ plus $n-1$. The smallest possible length of any sequence $uB(k,n)$ is therefore 
$$
l(k,n) = (k^n+k^{\lceil n/2 \rceil} + 2n - 2)/2.
$$

We refer to unoriented de Bruijn sequences $uB(1,n)$ and $uB(k,1)$ as trivial unoriented de Bruijn sequences and note that all trivial de Bruijn sequences have optimal length. The purpose of this paper is to determine for which pairs $(k,n)$ nontrivial unoriented de Bruijn sequences of optimal length exist and to construct unoriented de Bruijn sequences by way of Eulerian paths in Eulerizations of unoriented de Bruijn graphs. 

%%%%%%%%%%%%%%%%%%%%%%%%%%%%%%%%%%%%%
\section{Unoriented de Bruijn graphs}
%%%%%%%%%%%%%%%%%%%%%%%%%%%%%%%%%%%%%

The proof that de Bruijn sequences $B(k,n)$ exist for all $k,n$ begins by forming a \textit{$(k,n)$-de Bruijn graph}, $Bg(k,n)$.  $Bg(k,n)$ is a directed graph whose $k^{n-1}$ vertices are labeled by the words of length $n-1$ on $k$ symbols, and each of whose $k^n$ directed edges connect a subword to a potential consecutive subword in a sequence.  That is, there is an edge $\bm{v_1} \rightarrow \bm{v_2}$, if  $\bm{v_1} = v_1 v_2 \cdots v_{n-1}$ and $\bm{v_2} = v_2 \cdots v_{n-1} w$ for some $w \in \Sigma_k$.  Following an Eulerian circuit -- a path in the graph that visits each edge exactly once and starts and ends on the same vertex -- generates a de Bruijn sequence $B(k,n)$. 

Examples of de Bruijn graphs appear in Fig.~\ref{fig:dB}~(a-c).  In Fig.~\ref{fig:dB2_3}, the path 
$$00 \rightarrow 00 \rightarrow 01 \rightarrow 10 \rightarrow 01 \rightarrow 11 \rightarrow 11\rightarrow 10 \rightarrow 00$$
corresponds to the subwords 
$$000 \rightarrow 001 \rightarrow 010 \rightarrow 101 \rightarrow 011 \rightarrow 111 \rightarrow 110\rightarrow 100$$
that form the de Bruijn sequence 0001011100.  Since every vertex in $Bg(k,n)$ has even degree $2k$, there is an Eulerian circuit in any $Bg(k,n)$, and therefore a de Bruijn sequence $B(k,n)$ exists for any pair $(k,n)$. 

%%%%%%%%%%%%%%%%%%%%%%%%%%%%%%%%%
%Begin dB original and unoriented figures%
%%%%%%%%%%%%%%%%%%%%%%%%%%%%%%%%%

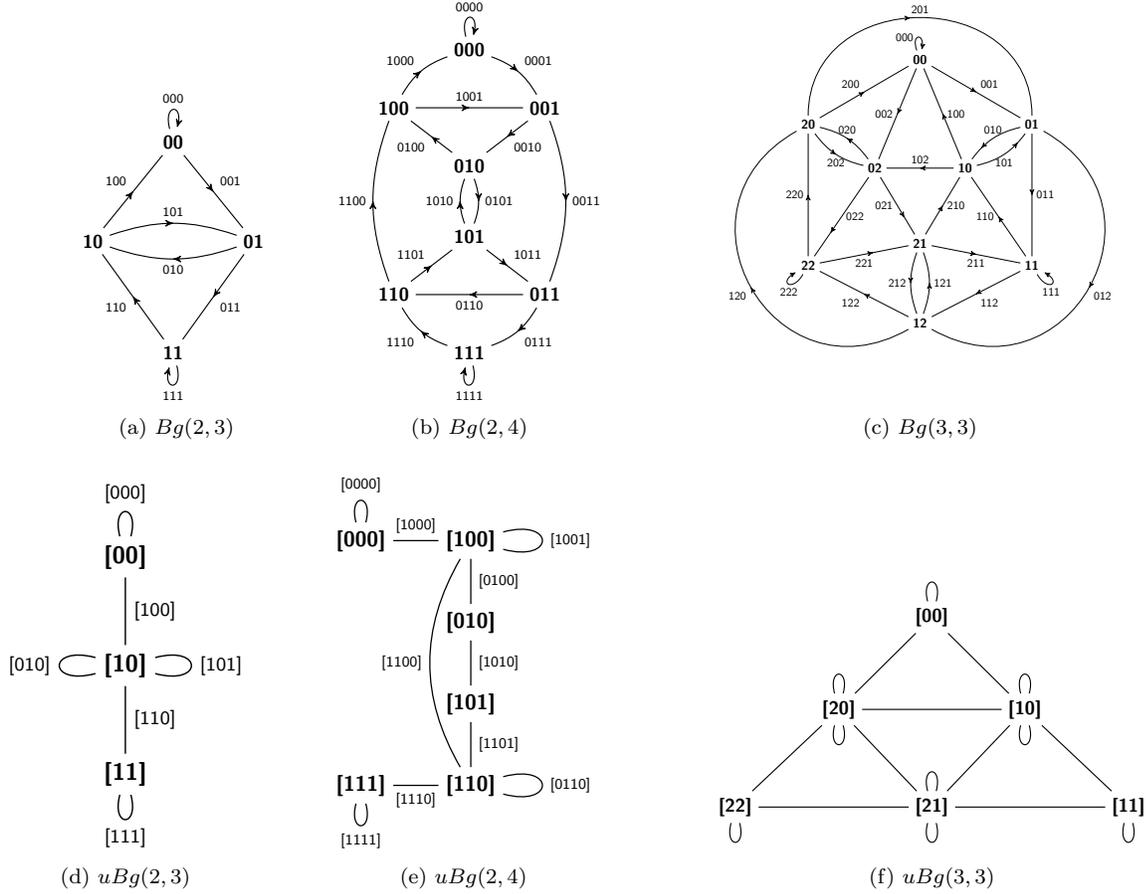
\begin{figure}[H]
\centering
      \subfloat[][$Bg(2,3)$] {
      \scalebox{0.7}{
         \begin{tikzpicture}[>=stealth',
         auto,
         node distance=1.25cm,
         bend angle=19,
         semithick,
         main node/.style={font=\sffamily\bfseries}]
         \tikzset{->-/.style={decoration={markings, mark=at position #1 with {\arrow[scale=1.5]{>}}},postaction={decorate}}}

  \node[main node] (1)                                 {00};
  \node[main node] (2) [below left=of 1,yshift=-.5cm]  {10};
  \node[main node] (3) [below =of 1,yshift=-2.25cm]    {11};
  \node[main node] (4) [below right=of 1,yshift=-.5cm] {01};

  \path[every node/.style={font=\sffamily\scriptsize}]
  
    (1) edge [->- =.5]           node  {001} (4)
        edge [loop above]        node  {000} (1)
    (2) edge [->- =.5]           node  {100} (1)
        edge [bend left,->- =.5] node  {101} (4)
    (3) edge [->- =.5]           node  {110} (2)
        edge [loop below]        node  {111} (3)
    (4) edge [->- =.5]           node  {011} (3)
        edge [bend left,->- =.5] node  {010} (2);
        \end{tikzpicture}
        \label{fig:dB2_3}
     }
   }
   \quad
      \subfloat[][$Bg(2,4)$] {
      \scalebox{0.70}{
         \begin{tikzpicture}[>=stealth',
         auto,
         node distance=.85cm,
         bend angle=20,
         semithick,
         main node/.style={font=\sffamily\bfseries}]
         \tikzset{->-/.style={decoration={markings, mark=at position #1 with {\arrow[scale=1.5]{>}}},postaction={decorate}}}

  \node[main node] (1)                    {000};
  \node[main node] (2) [below left=of 1]  {100};
  \node[main node] (3) [below right=of 2] {010};
  \node[main node] (4) [above right=of 3] {001};
  \node[main node] (5) [below =of 3]      {101};
  \node[main node] (6) [below left=of 5]  {110};
  \node[main node] (7) [below right=of 6] {111};
  \node[main node] (8) [above right=of 7] {011};

  \path[every node/.style={font=\sffamily\scriptsize}]
  
    (1) edge [bend left,->- =.5]   node  {0001} (4)
        edge [loop above]          node  {0000} (1)
        
    (2) edge [bend left,->- =.5]   node  {1000} (1)
        edge [->- =.5]             node  {1001} (4)
        
    (3) edge [->- =.5]             node  {0100} (2)
        edge [bend left,->- =.5]   node  {0101} (5)
        
    (4) edge [->- =.5]             node  {0010} (3)
        edge [bend left,->- =.5]   node  {0011} (8)
        
    (5) edge [->- =.5]             node  {1011} (8)
        edge [bend left,->- =.5]   node  {1010} (3)
        
    (6) edge [bend left,->- =.5]   node  {1100} (2)
        edge [->- =.5]             node   {1101} (5)
        
    (7) edge [bend left,->- =.5]   node  {1110} (6)
        edge [loop below]          node  {1111} (8)
        
    (8) edge [bend left,->- =.5]   node  {0111} (7)
        edge [->- =.5]             node  {0110} (6);
        \end{tikzpicture}
        \label{fig:dB2_4}
      }
   }
   \quad
      \subfloat [$Bg(3,3)$] {
      \scalebox{0.48}{
         \begin{tikzpicture}[>=stealth',
         auto,
         node distance=1cm,
         bend angle=20,
         semithick,
         main node/.style={font=\sffamily\bfseries}]
         \tikzset{->-/.style={decoration={markings, mark=at position #1 with {\arrow[scale=1.5]{>}}},postaction={decorate}}}

  \node[main node] (00)                                                {00};
  \node[main node] (20) [below left=of 00,xshift=-1.75cm,yshift=-.6cm] {20};
  \node[main node] (01) [below right=of 00,xshift=1.75cm,yshift=-.6cm] {01};
  \node[main node] (02) [below right=of 20, xshift=.5cm]               {02};
  \node[main node] (10) [below left=of 01, xshift=-.5cm]               {10};
  \node[main node] (21) [below =of 00,yshift=-3.6cm]                   {21};
  \node[main node] (22) [below left=of 21,xshift=-1.75cm,yshift=.6cm]  {22};
  \node[main node] (12) [below =of 21,yshift=-.7cm]                    {12};
  \node[main node] (11) [below right=of 21,xshift=1.75cm, yshift=.6cm] {11};

  \path[every node/.style={font=\sffamily\small}]
  
    (00) edge [in=80,out=110, loop]  node [left] {000} (00)
         edge [->- =.5]              node        {001} (01)
         edge [->- =.5]              node [left] {002} (02)
        
    (20) edge [->- =.5]                             node         {200} (00)
         edge [bend right, ->- =.5]                 node [below] {202} (02)
         edge [bend left=90,looseness=1.5,->- =.47] node         {201} (01)
        
    (01) edge [->- =.5]                             node         {011} (11)
         edge [bend right,->- =.75]                 node [above] {010} (10)
         edge [bend left=90,looseness=1.75,->- =.5] node         {012} (12)        
        
    (02) edge [bend right,->- =.5]   node [above] {020} (20)
         edge [->- =.7]              node [left]  {021} (21)
         edge [->- =.7]              node [right] {022} (22)
        
    (10) edge [->- =.5]              node [above] {102} (02)
         edge [bend right,->- =.8]   node [below] {101} (01)
         edge [->- =.5]              node [right] {100} (00)
        
    (21) edge [bend right,->- =.5]   node [left]  {212} (12)
         edge [->- =.5]              node [right] {210} (10)
         edge [->- =.5]              node [below] {211} (11)
        
    (22) edge [->- =.6]              node  [below] {221} (21)
         edge [->- =.5]              node          {220} (20)
         edge [in=210,out=240,loop]  node [below]  {222} (22)
        
    (12) edge [bend right,->-=.5]                   node [right] {121} (21)
         edge [bend left=90,looseness=1.75,->- =.5] node {120} (20)
         edge [->- =.5]                             node  {122} (22)
        
    (11) edge [->- =.5]              node [left] {110} (10)
         edge [->- =.5]              node        {112} (12)
         edge [in=-30,out=-60,loop]  node [below]{111} (11);
        \end{tikzpicture}
        \label{fig:dB3_3}
      }
   } 

\hspace{-1.5cm}
     \subfloat [$uBg(2,3)$] {
      \scalebox{0.9}{
      \begin{tikzpicture}[>=stealth',
      auto,
      node distance=1cm,
      bend angle=20,
      semithick,
      main node/.style={font=\sffamily\bfseries}, 
      every loop/.style={}]

  \node[main node] (1)                {[00]};
  \node[main node] (2) [below =of 1]  {[10]};
  \node[main node] (3) [below =of 2]  {[11]};

  \path[every node/.style={font=\sffamily\footnotesize}]
  
    (1) edge               node  {[100]} (2)
        edge [loop above]  node  {[000]} (1)
    (2) edge               node  {[110]} (3)
        edge [loop left ]  node  {[010]} (2)
        edge [loop right]  node  {[101]} (2)
    (3) edge [loop below]  node  {[111]} (3);
      \end{tikzpicture}
      \label{fig:udB2_3}
     }
   } %&
   \hspace{.5cm}
      \subfloat [$uBg(2,4)$] {
      \scalebox{0.8}{
      \begin{tikzpicture}[>=stealth',
      auto,
      node distance=.75cm,
      semithick,
      main node/.style={font=\sffamily\bfseries},
      every loop/.style={}] 

  \node[main node] (000)                 {[000]};
  \node[main node] (100) [right =of 000] {[100]};
  \node[main node] (010) [below =of 100] {[010]};
  \node[main node] (101) [below =of 010] {[101]};
  \node[main node] (110) [below =of 101] {[110]};
  \node[main node] (111) [left =of 110]  {[111]};

  \path[every node/.style={font=\sffamily\scriptsize}]
  
    (000) edge               node         {[1000]} (100)
          edge [loop above]  node         {[0000]} (000)
    (100) edge               node [right] {[0100]} (010)
          edge [bend right]  node [left]  {[1100]} (110)
          edge [loop right]  node         {[1001]} (100)
    (010) edge               node [right] {[1010]} (101)
    (101) edge               node [right] {[1101]} (110)
    (110) edge               node         {[1110]} (111)
          edge [loop right]  node         {[0110]} (110)
    (111) edge [loop below]  node         {[1111]} (11);
      \end{tikzpicture}
      \label{fig:udB2_4}
      }
   }
   \hspace{1cm}
      \subfloat [$uBg(3,3)$] {
      \scalebox{0.7}{
      \begin{tikzpicture}[>=stealth',
      auto,
      node distance=1.25cm,
      semithick,
      main node/.style={font=\sffamily\bfseries},
      every loop/.style={}] 

  \node[main node] (00)                                  {[00]};
  \node[main node] (20) [below left=of 00,yshift=-.26cm] {[20]};
  \node[main node] (01) [below right=of 00,yshift=-.26cm]{[10]};
  \node[main node] (21) [below =of 00,yshift=-1.75cm]     {[21]};
  \node[main node] (22) [left=of 21,xshift=-1.6cm]       {[22]};
  \node[main node] (11) [right=of 21,xshift=1.6cm]       {[11]};

  \path[every node/.style={font=\sffamily\scriptsize}]
  
    (00) edge               node  [left]{} (20)
         edge [loop above]  node  {} (00)
    (20) edge               node  {} (01)
         edge               node  [left]{} (22)
         edge [loop above]  node  [left]{} (20)
         edge [loop below]  node  [below]{} (20)
    (01) edge               node  [right]{} (00)
         edge               node  {} (11)
         edge [loop above]  node  [right]{} (01)
         edge [loop below]  node  {} (01)
    (21) edge               node  {} (20)
         edge               node  {} (01)
         edge [loop above]  node  [left]{} (21)
         edge [loop below]  node  {} (21)
    (22) edge               node  [below]{} (21)
         edge [loop below]  node  {} (22)
    (11) edge [loop below]  node  {} (11)
         edge               node  {} (21);
      \end{tikzpicture}
      \label{fig:udB3_3}
      }
   }
   \caption{(a-c) de Bruijn graphs $Bg(k,n)$. Vertices represent length-$(n-1)$ words and edges represent length-$n$ words. (d-f) unoriented de Bruijn graphs $uBg(k,n)$. Vertices represent length-$(n-1)$ reflected pairs, and edges represent length-$n$ reflected pairs.}
   \label{fig:dB}
\end{figure}
%%%%%%%%%%%%%%
%End figure%
%%%%%%%%%%%%%%

An unoriented de Bruijn sequence of optimal length would be formed by following a path in a de Bruijn graph that traverses exactly one of the two edges corresponding to each reflected pair.  The main result of this paper is that, if neither $k$ nor $n$ equals 1, such paths only exist when both $k$ and $n$ are no greater than $3$  (Thm.~\ref{flip}).  Towards this result, we construct unoriented de Bruijn graphs $uBg(k,n)$ whose vertices are in 1-1 correspondence with the length-$(n-1)$ reflections on $k$ symbols, and whose edges are in 1-1 correspondence with the reflections of length $n$ on $k$ symbols.  To illustrate this construction, consider the de Bruijn graph $Bg(2,4)$ in Fig.~\ref{fig:dB2_4}.  The edge $001 \rightarrow 011$ represents the word ${\bm e} = 0011$ whose reflection ${\bm e'} = 1100$ is represented by the edge $110 \rightarrow 100$.  This is, in general, the case.

%%%%%%%%%%%%%%%%%%%%%%%%%%
\begin{lemma}
\label{undir}
For every edge $\bm{v} \rightarrow \bm{w}$ in the graph $Bg(k,n)$, there exists an edge $\bm{w'} \rightarrow \bm{v'}$ in $Bg(k,n)$. Moreover, the length-$n$ words represented by these edges form a reflected pair. 
\end{lemma}
%%%%%%%%%%%%%%%%%%%%%%%%%%

\begin{proof}
Let $\bm{v}=v_1\cdots v_{n-1}$.  Consider any  $\bm{w}$ such that $\bm{v} \rightarrow \bm{w}$. Then, $\bm{w}=v_2\cdots v_{n-1}w$ and $\bm{w'}=w v_{n-1}\cdots v_2$ for some $w \in \Sigma_k$. Since the last $n-2$ elements of $\bm{w'}$, $v_{n-1}\cdots v_2$, agree with the first $n-2$ elements of $\bm{v'}=v_{n-1}\cdots v_1$, there is an edge $\bm{w'} \rightarrow \bm{v'}$.  The edge $\bm{v} \rightarrow \bm{w}$ represents the word $\bm{e} = v_1 \cdots v_{n-1} w$, and the edge $\bm{w'} \rightarrow \bm{v'}$ represents the reflection $\bm{e'} = w v_{n-1} \cdots v_1$.  

\qed
\end{proof}

Lemma \ref{undir} allows for the definition of a graph whose vertices represent the reflected pairs of words of length $n-1$ and whose undirected edges, $\bm{[v]} \leftrightarrow \bm{[w]}$, represent the two directed edges $\bm{v} \rightarrow \bm{w}$ and $\bm{w'} \rightarrow \bm{v'}$.  Essential to the proof of Lemma \ref{undir} and the definition of such a graph is the consideration of the first and last $n-2$ characters of a word of length $n-1$.  We define the \textit{prefix} (\textit{suffix}) of a word of length $m$ to be the first (last) $m-1$ letters of that word.  

%%%%%%%%%%%%%%%%%%%%%%%%%%
\begin{definition}
For positive integers $k,n$, an \textit{unoriented de Bruijn graph} $uBg(k,n)$ contains $(k^n+k^{\lceil n/2 \rceil})/2$ vertices labeled by the equivalence classes $\bm{[v]}$ of $k$-ary words of length $n-1$.  Edges of $uBg(k,n)$ are labeled by distinct equivalence classes of words of length $n$.  An undirected edge labeled $[\bm{e}]$ connects the vertices labeled by the equivalence classes of the prefix and suffix of $\bm{e}$, respectively.  Thus, there is an (undirected) edge $[\bm{v}] \leftrightarrow [\bm{w}]$ if $\bm{v} \rightarrow \bm{w}$ or $\bm{v} \rightarrow \bm{w'}$ is a (directed) edge in the de Bruijn graph $Bg(k,n)$. 
\end{definition}
%%%%%%%%%%%%%%%%%%%%%%%%%%

\noindent  Examples of unoriented de Bruijn graphs are shown in Fig.~\ref{fig:dB}~(d-f).  If there exists an edge $\bm{v} \rightarrow \bm{v'}$ in the graph $Bg(k,n)$, then there exists a undirected loop $\hookrightarrow \bm{[v]}$ in the graph $uBg(k,n)$. For example, the edges $10 \rightarrow 01$, and $01 \rightarrow 10$ in $Bg(2,3)$ correspond to the two loops $\hookrightarrow [10]$ in the graph $uBg(2,3)$ (Fig.~\ref{fig:udB2_3}).  These loops represent the classes [101] and [010]. \\

Our definition of an unoriented de Bruijn graph differs from that of the undirected de Bruijn graph defined by Esfahanian and Hakimi \cite{eh} and investigated by various authors \cite{kf,lz}.  Their undirected de Bruijn graphs are constructed by replacing directed edges in de Bruijn graphs by undirected edges and then removing all loops and multiple edges.

%%%%%%%%%%%%%%%%%%%%%%%%%%%%%%%%%%%%%%%%%%%%%%%%%%%%%%
\section{Generating unoriented de Bruijn sequences}
%%%%%%%%%%%%%%%%%%%%%%%%%%%%%%%%%%%%%%%%%%%%%%%%%%%%%%

\noindent As an unoriented de Bruijn graph $uBg(k,n)$ is traversed, the direction taken along each edge determines the sequence of subwords in the corresponding unoriented de Bruijn sequence. In Fig.~\ref{fig:udB2_3}, if the undirected edge from $[00]$ to $[10]$ is traversed $[00]\rightarrow [10]$, the corresponding subword is 001. If the same edge is traversed $[10]\rightarrow [00]$, the subword is 100. Continuing this example, the path 
$$[00] \hookrightarrow [00] \rightarrow [10] \hookrightarrow [10]\hookrightarrow [10] \rightarrow [11] \hookrightarrow [11]$$
corresponds to the subwords 
$$000 \rightarrow 001 \rightarrow 010 \rightarrow 101 \rightarrow 011 \rightarrow 111,$$
and the $(2,3)$-unoriented de Bruijn sequence is 00010111. However, not every path corresponds to a valid sequence. For instance, in $uBg(2,3)$ the path 
$$[00] \hookrightarrow [00] \rightarrow [10] \hookrightarrow [10] \rightarrow [11]$$
is a valid path but does not correspond to a valid sequence since the subword 011 cannot follow 010.

For any vertex ${[\bm{v}]}$ in $uBg(k,n)$, it is necessary to distinguish the edges incident to ${[ \bm{v}]}$ that correspond to edges ${\bm v} \rightarrow {\bm w}$ and ${\bm w'} \rightarrow {\bm v'}$ in $Bg(k,n)$ from those that correspond to edges ${\bm v'} \rightarrow {\bm w}$ and ${\bm w'} \rightarrow {\bm v}$ in $Bg(k,n)$. 

%%%%%%%%%%%%%%%%
\begin{definition}
The incidence of an edge $[\bm{e}]$ to a vertex $[\bm{v}]$ is said to be of \textit{Type I (II)} if either $\bm{v}$ is the prefix (suffix) of $\bm{e}$, or $\bm{v'}$ is the suffix (prefix) of $\bm{e}$.  We will also say that an edge $[\bm{e}]$ is of \textit{Type I (II) relative to $[\bm{v}]$} if the incidence of $[\bm{e}]$ to $[\bm{v}]$ is of Type I (II).
\end{definition} 
%%%%%%%%%%%%%%%%

Examples of unoriented de Bruijn graphs in which edges are labeled according to their Type relative to each vertex, shown in Fig.~\ref{fig:udB}, illustrate that an edge is not necessarily of the same type relative to the two vertices to which it is incident.  For example, the edge [110] in $uBg(2,3)$ is of Type 1 relative to [11] because 11 is the prefix of 110, but is of Type II relative to [10] because 10 is the suffix of 110.  Also note if $\bm{v} = \bm{v'}$, then all edges incident to $\bm{v}$ are of both types relative to $\bm{v}$ and a loop contributes two to the count of incidences of Type I or II to any vertex.  For example, the vertex $[10]$ in the graph $uBg(2,3)$ shown in Fig.~\ref{fig:udB}~(a) has three incidences of Type I and three of Type II.  

In order to generate an unoriented de Bruijn sequence from $uBg(k,n)$, one must traverse the graph by entering each vertex $[{\bm v}]$ on an edge of one type relative to $[{\bm v}]$ and leaving on an edge of the other type relative to $[{\bm v}]$.  There is effectively no restriction on which edges can be traversed from a vertex $\bm{v}$ such that $\bm{v} = \bm{v'}$ since all incidences to these vertices are of both types.  We define an \textit{alternating Eulerian path} in an unoriented de Bruijn graph to be an Eulerian path that satisfies this criterion: 

%%%%%%%%%%%%%%%%
\begin{definition}
\label{aEp}
An \textit{alternating Eulerian path} in an unoriented de Bruijn graph is an Eulerian path such that if it enters a vertex on an edge of Type I (II) relative to that vertex, then it leaves that vertex on an edge of Type II (I) relative to that vertex.
\end{definition} 
%%%%%%%%%%%%%%%%

The relationship between unoriented de Bruijn sequences and graphs is thus more complex than that between original de Bruijn sequences and graphs. The existence of an unoriented de Bruijn sequence $uB(k,n)$ of optimal length implies the existence of an Eulerian path in the unoriented de Bruijn graph $uBg(k,n)$ but the converse statement is more restrictive: the existence of an alternating Eulerian path in $uBg(k,n)$ implies the existence of a sequence $uB(k,n)$ of optimal length.

Algorithms such as Fleury's algorithm \cite{Fleury} and Hierholzer's algorithm \cite{Hierholzer} are guaranteed to produce an Eulerian circuit in a graph with no vertices of odd degree or an Eulerian path in a graph with exactly two vertices of odd degree.  These algorithms may be readily adapted to produce alternating circuits or paths in graphs with zero or two odd-degree vertices that also satisfy additional conditions, which are provided in Proposition~\ref{alt_conditions}.

\newpage

%%%%%%%%%%%%%%%%%%%%%%%%%
% Begin Colored graphs (small k,n)
%%%%%%%%%%%%%%%%%%%%%%%%%

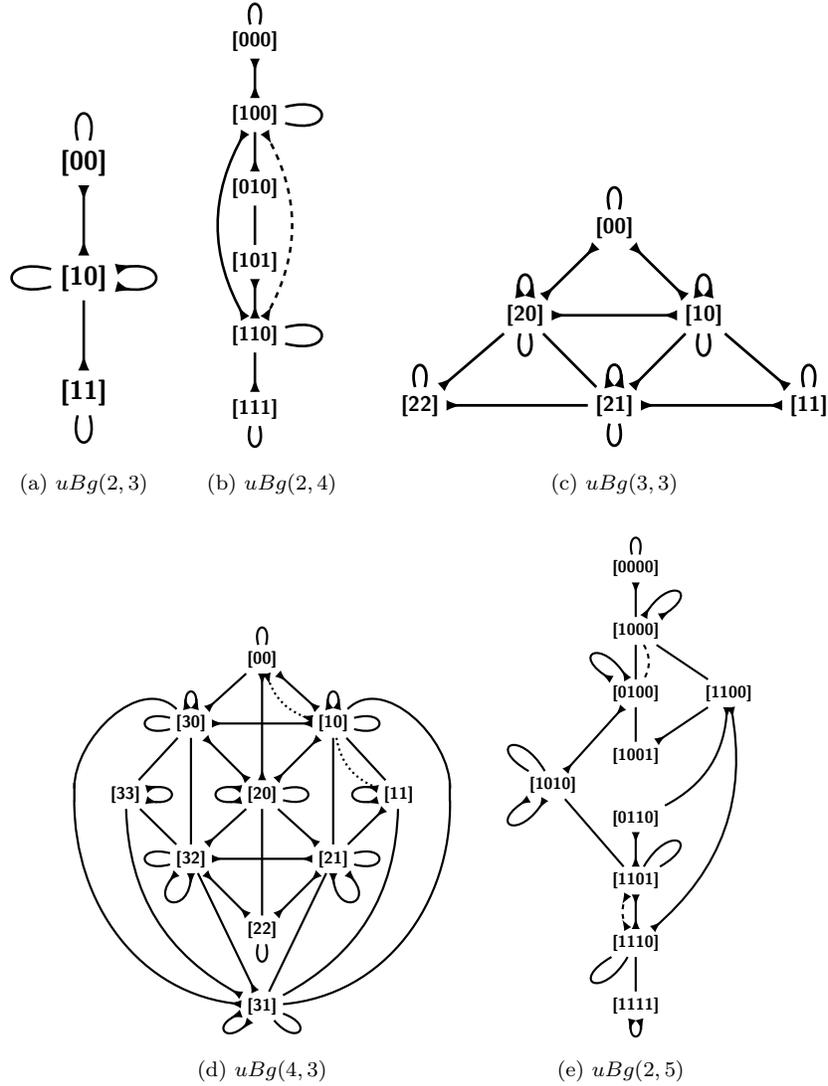
\begin{figure}[H]
\centering
      \subfloat [$uBg(2,3)$] {
      \resizebox{0.175\textwidth}{!}{
      \begin{tikzpicture}[>=stealth',
      auto,
      node distance=.95cm,
      bend angle=20,
      semithick,
      main node/.style={font=\sffamily\bfseries},every loop/.style={line width=1pt}, 
      red/.style={line width=1pt,color=black,>=latex,>-<}, 
      blue/.style={line width=1pt, color=black}]

  \node[main node] (1)                {[00]};
  \node[main node] (2) [below =of 1]  {[10]};
  \node[main node] (3) [below =of 2]  {[11]};

  \path[every node/.style={font=\sffamily\scriptsize}]
  
    (1) edge [red]             node  {} (2)
        edge [blue,loop above]  node  {} (1)
        
    (2) edge [blue, >=latex,-<]node  {} (3)
        edge [red,loop right]  node  {} (2)
        edge [blue,loop left]  node  {} (2)
        
    (3) edge [blue,loop below]  node  {} (3);
      \end{tikzpicture}
      \label{fig:color_udB2_3}
     }
   }
      \subfloat [$uBg(2,4)$] {
      \resizebox{0.125\textwidth}{!}{
      \begin{tikzpicture}[>=stealth',
      auto,
      node distance=.7cm,
      semithick,
      main node/.style={font=\sffamily\bfseries}, every loop/.style={line width=1.3pt,color=black},
      red/.style={line width=1.3pt,color=black,>=latex,>-<}, 
      blue/.style={line width=1.3pt, color=black}] 

  \node[main node] (1)               {[000]};
  \node[main node] (2) [below =of 1] {[100]};
  \node[main node] (3) [below =of 2] {[010]};
  \node[main node] (5) [below =of 3] {[101]};
  \node[main node] (6) [below =of 5] {[110]};
  \node[main node] (7) [below =of 6] {[111]};

  \path[every node/.style={font=\sffamily\scriptsize}]
  
    (1) edge [red]                    node [right] {} (2)
        edge [blue, loop above]        node         {} (1)
        
    (2) edge [blue,>=latex,-<]        node [left]  {} (3)
        edge [red,bend right]         node         {} (6)
        edge [red,bend left, dashed]  node         {} (6)
        edge [loop right]             node         {} (2)
        
    (3) edge [blue]                    node [right] {} (5)
    
    (5) edge [red]                    node [right] {} (6)
    
    (6) edge  [blue,>=latex,-<]       node         {} (7)
        edge [loop right]             node         {} (6)
        
    (7) edge [blue, loop below]        node         {} (8);
      \end{tikzpicture}
      \label{fig:color_udB2_4}
      }
   }
   \quad
      \subfloat [$uBg(3,3)$] {
      \resizebox{0.4\textwidth}{!}{
      \begin{tikzpicture}[>=stealth',
      auto,
      node distance=.85cm,
      semithick,
      main node/.style={font=\sffamily\bfseries},every loop/.style={line width=1.3pt}, 
      red/.style={line width=1.3pt,color=black,>=latex,>-<}, 
      blue/.style={line width=1.3pt, color=black}, 
      br/.style={color=black,line width=1.3pt,>=latex,>-},  
      brr/.style={color=black,line width=1.3pt,>=latex,-<}]

  \node[main node] (00)                                {[00]};
  \node[main node] (20) [below left=of 00,yshift=-.25cm] {[20]};
  \node[main node] (01) [below right=of 00,yshift=-.25cm] {[10]};
  \node[main node] (21) [below =of 00,yshift=-1.5cm]  {[21]};
  \node[main node] (22) [left=of 21,xshift=-1.5cm] {[22]};
  \node[main node] (11) [right=of 21,xshift=1.5cm]   {[11]};

  \path[every node/.style={font=\sffamily\scriptsize}]
  
    (00) edge [red]              node  {} (20)
         edge [blue,loop above]   node  {} (00)
    (20) edge [red]              node  {} (01)
         edge [brr]              node  {} (22)
         edge [red,loop above]   node  {} (20)
         edge [blue,loop below]  node  {} (20)
    (01) edge [red]              node  {} (00)
         edge [brr]              node  {} (11)
         edge [red,loop above]   node  {} (01)
         edge [blue,loop below]  node  {} (01)
    (21) edge [blue]             node  {} (20)
         edge [br]               node  {} (01)
         edge [red,loop above]   node  {} (21)
         edge [blue,loop below]  node  {} (21)
    (22) edge [br]               node  {} (21)
         edge [blue,loop above]   node  {} (22)
    (11) edge [blue,loop above]   node  {} (11)
         edge [red]              node  {} (21);
      \end{tikzpicture}
      \label{fig:color_udB3_3}
      }
   }
   \quad
      \subfloat [$uBg(4,3)$] {
      \resizebox{0.35\textwidth}{!}{
      \begin{tikzpicture}[>=stealth', 
      auto,
      node distance=.5cm,
      semithick,
      main node/.style={font=\sffamily\bfseries},
      every loop/.style={line width=1.3pt}, 
      red/.style={line width=1.3pt,color=black,>=latex,>-<}, 
      blue/.style={line width=1.3pt, color=black}, 
      br/.style={color=black,line width=1.3pt,>=latex,-<},  
      brr/.style={color=black,line width=1.3pt,>=latex,>-}] 
      
  \draw[line width=1.35pt] (3.96,-2.61) -- (3.96,-3.0);
  \draw[line width=1.35pt] (-3.96,-2.61) -- (-3.96,-3.0);
  
  \node[main node] (00)                                            {[00]};%
  \node[main node] (30) [below =of 00,yshift=-.25cm,xshift=-1.5cm] {[30]};%
  \node[main node] (01) [below =of 00,yshift=-.25cm,xshift=1.5cm]  {[10]};%
  \node[main node] (20) [below =of 00,yshift=-1.75cm]              {[20]};%
  \node[main node] (33) [left=of 20,xshift=-1.5cm]                 {[33]};%
  \node[main node] (33i) [left=of 33, xshift=-.01cm]               {};%%%%%%
  \node[main node] (11) [right=of 20,xshift=1.5cm]                 {[11]};%
  \node[main node] (11i) [right=of 11,xshift=0.01cm]               {};%%%%%
  \node[main node] (32) [below=of 20,yshift=-.25cm,xshift=-1.5cm]  {[32]};%
  \node[main node] (21) [below =of 20,yshift=-.25cm,xshift=1.5cm]  {[21]};%
  \node[main node] (22) [below =of 20,yshift=-1.75cm]              {[22]};%
  \node[main node] (31) [below =of 22,yshift=-.5cm]                {[31]};%

  \path[every node/.style={font=\sffamily\scriptsize}]
  
    (00) edge [red]                   node  {} (20)
         edge [blue,loop above]        node  {} (00)
         edge [br]                   node  {} (30)
         
    (30) edge [red]                   node  {} (01)
         edge [red,loop above]        node  {} (30)
         edge [black,loop left]       node  {} (30)
         edge [blue]                  node  {} (32)
         edge [blue,in=90,out=135, shorten >=0pt]  node {} (33i)  %%%%%%%%%%%%%%%%%%
         
    (01) edge [red]                   node  {} (00)
         edge [red,dotted,bend left]  node  {} (00)
         edge [blue]                  node  {} (11)
         edge [blue,dotted,bend right]node  {} (11)
         edge [red,loop above]        node  {} (01)
         edge [blue,loop right]       node  {} (01)
         edge [red]                   node  {} (20)
         edge [blue]                  node  {} (21)
         edge [blue,in=89,out=45]     node  {} (11i) %%%%%%%%%%%%%%%%% 
         
    (20) edge [br]                    node  {} (21)
         edge [br]                    node  {} (32)
         edge [blue]                  node  {} (22)
         edge [red]                   node  {} (30)
         edge [red,loop left]         node  {} (20)
         edge [blue,loop right]       node  {} (20)
         
    (33) edge [blue]                  node {} (30)
         edge [red,loop right]       node {} (33)
         edge [br,bend right]         node {} (31)
         edge [blue]                  node {} (32)
         
    (32) edge [red]                   node {} (22)
         edge [blue,loop left]        node {} (32)
         edge [red,in=225,out=270,loop] node {} (32)
         edge [red]                   node {} (21)
         edge [br]                    node {} (31)
         
    (21) edge [br]                         node {} (11)
         edge [red,in=270,out=315,loop]      node {} (21)
         edge [blue,loop right]              node {} (21)
         edge [blue]                         node {} (31)
         
    (22) edge [red]                    node  {} (21)
         edge [blue,loop below]       node  {} (22)
         
    (31) edge [red,in=230,out=200,loop]      node {} (31)
         edge [blue,in=310,out=340,loop]     node {} (31)
         edge [brr,in=270,out=175]           node {} (33i)  %%%%%%%
         edge [blue,in=271,out=5]            node {} (11i)  %%%%%%%%%%%%%%%%%%
    (11) edge [red,loop left]               node {} (11)
         edge [blue, bend left]              node {} (31);

      \end{tikzpicture}
      \label{fig:color_udB4_3}
      }
   }   
 %  \quad
 \subfloat [$uBg(2,5)$] {
      \resizebox{0.25\textwidth}{!}{
      \begin{tikzpicture}[>=stealth',   
      auto,
      node distance=.75cm,
      semithick,
      main node/.style={font=\sffamily\bfseries},
      every loop/.style={line width=1.3pt}, 
      red/.style={line width=1.3pt, color=black,>=latex,>-<}, 
      blue/.style={line width=1.3pt, color=black}] 

  \node[main node] (0000)                  {[0000]};
  \node[main node] (1000) [below =of 0000] {[1000]};
  \node[main node] (0100) [below =of 1000] {[0100]};
  \node[main node] (1001) [below =of 0100] {[1001]};
  \node[main node] (0110) [below =of 1001] {[0110]};
  \node[main node] (1101) [below =of 0110] {[1101]};
  \node[main node] (1110) [below =of 1101] {[1110]};
  \node[main node] (1111) [below =of 1110] {[1111]};
  \node[main node] (1010) [below left =of 1001, yshift=.5cm] {[1010]};
  \node[main node] (1100) [right =of 0100] {[1100]};

  \path[every node/.style={font=\sffamily\scriptsize}]
  
    (0000) edge [blue,loop above]           node {} (0000)
    (1000) edge [blue]                     node {} (1100)
           edge [blue]                     node {} (0100)
           edge [blue,dashed,bend left]    node {} (0100)
           edge [red,in=30,out=60,loop]    node {} (1000)
           edge [blue, >=latex,-<]		node {} (0000)
    (0100) edge [blue]                     node {} (1001)
           edge [red]                      node {} (1010)
           edge [red,in=120,out=150,loop]  node {} (0100)
    (1001) edge [blue,>=latex,>-]                     node {} (1100)
    (0110) edge [blue,bend right,>=latex,-<]          node {} (1100)
           edge [red]                      node {} (1101)
    (1101) edge [blue]                     node {} (1010)
           edge [blue,in=30,out=60,loop]   node {} (1101)
           edge [red]                      node {} (1110)
           edge [red,dashed,bend right]    node {} (1110)
    (1110) edge [blue,in=210,out=240,loop] node {} (1110)
           edge [blue]                     node {} (1111)
    (1111) edge [red,loop below]          node {} (1111)
    (1010) edge [red,in=210,out=240,loop]  node {} (1010)
           edge [blue,in=120,out=150,loop] node {} (1010)
    (1100) edge [red,bend left,>=latex,>-<] node {} (1110);
      \end{tikzpicture}
      \label{fig:color_udB2_5}
      }
   }
   \caption{Unoriented de Bruijn graphs $uBg(k,n)$.   An edge of Type I  relative to a vertex is marked by a triangle, and an edge of Type II is unmarked.  Although edges incident to palindromic reflected pairs can be of either type, we show here choices which ensure that the graph supports an alternating Eulerian path. The dashed edges in subfigures (b,d,e) are duplicates of edges in the graphs that have been inserted in such a way that each of the graphs contains an alternating Eulerian path that generates an unoriented de Bruijn sequence. }
   \label{fig:udB}
\end{figure}

\begin{proposition}
\label{alt_conditions}
An unoriented de Bruijn sequence $uB(k,n)$ of optimal length exists if and only if the following conditions on $uBg(k,n)$ are satisfied:

	\begin{enumerate}
	\item The number of odd-degree vertices is zero or two.
	\item The numbers of Type-I and Type-II incidences 
	relative to every even-degree vertex representing a non-palindromic reflected pair are equal.
	\item The numbers of Type-I and Type-II incidences relative to every 
	odd-degree vertex representing a non-palindromic reflected pair differ by one.
	\end{enumerate}
	
\end{proposition}

\vspace{1mm}

In lieu of a proof of Proposition~\ref{alt_conditions}, we state a modification of Hierholzer's algorithm that produces an alternating circuit or path in a graph satisfying the conditions of the proposition.  First note that if the graph $uBg(k,n)$ has exactly two vertices of odd degree, then the addition of an edge $\overline{{\bm e}}$ between those two vertices results in an graph $\overline{uBg}(k,n)$ with only even-degree vertices.   Secondly, note that Conditions 2 and 3 only apply to vertices representing non-palindromic reflected pairs since incidence types are only defined for these vertices.  Choose a designation of types of incidence relative to vertices representing palindromic reflected pairs in such a way that all vertices in $\overline{uBg}(k,n)$ have an equal number of Type-I and Type-II incidences.  Upon finding an alternating Eulerian circuit in $\overline{uBg}(k,n)$, the edge $\overline{{\bm e}}$ may be removed to form an alternating Eulerian path in $uBg(k,n)$.  It therefore suffices to state an algorithm that finds an alternating Eulerian circuit in a graph that has only even-degree vertices and has an equal number of Type-I and Type-II incidences at every vertex.

Hierholzer's algorithm for finding Eulerian circuits, modified for finding alternating Eulerian circuits, consists of the following components:

\begin{enumerate}
\item Choose any vertex $[\bm{v}]$.  Follow an alternating path of edges starting at $[\bm{v}]$ until the path forms a circuit by returning to $[\bm{v}]$ on an edge of the opposite type relative to $[\bm{v}]$ from the initial edge in the path.  This is possible since every vertex has an equal number of Type-I and Type-II incidences, so if the path enters a vertex on an edge with a Type-I (II) incidence, then there is an edge with a Type-II (I) incidence to that vertex which is not yet part of the path.  

\item There may be a vertex $[\bm{w}]$ belonging to the current circuit $C$ which has incident edges that are not part of $C$. If so, form another alternating circuit $C'$ starting from and returning to $[\bm{w}]$ and not including any edges in $C$.  This is possible since each vertex in the graph $uBg(k,n)$ with the edges in $C$ deleted has even degree.  Choose the circuit $C'$ so that it starts (returns) on an edge of the opposite (same) type relative to $[\bm{w}]$ from (as) the edge that enters $[\bm{w}]$ in $C$.  Condition 2 ensures that this is possible -- if one happens to return to $[\bm{w}]$ on an edge that is of the same type relative to $[\bm{w}]$ as the starting edge, one may continue on until one returns again to $[\bm{w}]$ on an edge of the opposite type.  Splice $C'$ into $C$ to form a new current circuit.  The conditions of the types of edges that start and end $C'$ guarantee that the new current circuit is alternating.  

\item Repeat Step 2 until every vertex in the current circiut, $C$, only has incident edges which belong to $C$. The current circuit will then be an alternating Eulerian circuit. 
\end{enumerate}

\noindent What now remains in question is for what values of $k$ and $n$ the graph $uBg(k,n)$ satisfies the conditions of Proposition~\ref{alt_conditions}.  For small values of $k,n$, we observe the presence of either 0 or 2 odd-degree vertices and can find alternating Eulerian paths. Fig.~\ref{fig:udBpaths} provides by example the existence of unoriented de Bruijn sequences of optimal length for $k,n \leq 3$.

%%%%%%%%%%%%%%%%%%
%%Begin existing unoriented de Bruijn sequences (paths)
%%%%%%%%%%%%%%%%%%%
\begin{figure}[H]
\centering
      \subfloat[$uBg(2,2)$]{
      \resizebox{.055\textwidth}{!}{
      \begin{tikzpicture}[>=stealth',
      auto,
      node distance=1.5cm,
      bend angle=20,
      semithick,
      main node/.style={font=\sffamily\bfseries}]
      \tikzset{->-/.style={decoration={markings, mark=at position #1 with {\arrow[scale=1.3]{>}}},postaction={decorate}}}

  \node[main node] (0)                {[0]};
  \node[main node] (1) [below =of 0]  {[1]};
  
  \path[every node/.style={font=\sffamily\scriptsize}]
  
    (0) edge [->-=.5]      node  {} (1)
        edge [loop above]  node  {} (0)
    (1) edge [loop below]  node  {} (1);
      \end{tikzpicture}
      \label{fig:udB2_2path}
     }
   }
   \quad
   \subfloat[$uBg(3,2)$]{
     \scalebox{0.77}{
      \begin{tikzpicture}[>=stealth',
      auto,
      node distance=1.5cm,
      semithick,
      main node/.style={font=\sffamily\bfseries}] 
      \tikzset{->-/.style={decoration={markings, mark=at position #1 with {\arrow[scale=1.5]{>}}},postaction={decorate}}}
 
  \node[main node] (0)                     {[0]};
  \node[main node] (1) [below left=of 0]   {[1]};
  \node[main node] (2) [below right=of 0]  {[2]};

  \path[every node/.style={font=\sffamily\scriptsize}]
  
    (0) edge  [->-=.5]       node   {} (1)
        edge  [loop above]   node   {} (0)
    (1) edge [->-=.5]        node   {} (2)
        edge [loop below]    node   {} (1)
    (2) edge [loop below]    node   {} (2)
        edge [->-=.7]        node   {} (0);
      \end{tikzpicture}
      \label{fig:udB3_2path}
      }
   }
   \quad
   \subfloat [$uBg(2,3)$] {
      \resizebox{.14\textwidth}{!}{
      \begin{tikzpicture}[>=stealth', 
      auto, 
      node distance=1cm, 
      bend angle=20,
      semithick,
      main node/.style={font=\sffamily\bfseries}]
      \tikzset{->-/.style={decoration={markings, mark=at position #1 with {\arrow[scale=1.5]{>}}},postaction={decorate}}}

  \node[main node] (00)                 {[00]};
  \node[main node] (10) [below =of 00]  {[10]};
  \node[main node] (11) [below =of 10]  {[11]};

  \path[every node/.style={font=\sffamily\scriptsize}]
  
    (00) edge [->-=.5]      node  {} (10)
         edge [loop above]  node  {} (00)
    (10) edge [->-=.5]      node  {} (11)
         edge [loop left ]  node  {} (10)
         edge [loop right]  node  {} (10)
    (11) edge [loop below]  node  {} (11);
      \end{tikzpicture}
      \label{fig:udB2_3path}
     }
   } 
    \quad
    \subfloat [$uBg(3,3)$] {
      \scalebox{0.71}{
      \begin{tikzpicture}[>=stealth',
      auto,
      node distance=.5cm,
      semithick,
      main node/.style={font=\sffamily\bfseries}] 
       \tikzset{->-/.style={decoration={markings, mark=at position #1 with {\arrow[scale=1.5]{>}}},postaction={decorate}}}
  
  \node[main node] (00)                                   {[00]};
  \node[main node] (20) [below left=of 00,yshift=-.25cm]  {[20]};
  \node[main node] (01) [below right=of 00,yshift=-.25cm] {[10]};
  \node[main node] (21) [below =of 00,yshift=-1.5cm]      {[21]};
  \node[main node] (22) [left=of 21,xshift=-1.5cm]        {[22]};
  \node[main node] (11) [right=of 21,xshift=1.5cm]        {[11]};

  \path[every node/.style={font=\sffamily\scriptsize}]
  
    (00)edge  [loop above]   node         {} (00)
         edge  [->-=.6]       node [right] {} (01)
         
    (20) edge  [->-=.6]       node         {} (01)
         edge  [loop above]   node [left]  {} (20)
         edge  [loop below]   node         {} (20)
         edge  [->-=.6]       node [right]  {} (00)
         
    (01) edge  [->-=.6]       node [right] {} (11)
         edge  [loop above]   node [right] {} (01)
         edge  [loop below]   node         {} (01)
         edge  [->-=.6]       node [right] {} (21)
         
    (21) edge  [->-=.6]       node [left]  {} (20)
         edge  [loop above]   node         {} (21)
         edge  [loop below]   node         {} (21)
         edge  [->-=.6]       node         {} (22)
         
    (22) edge [loop below]    node [left]  {} (22)
         edge [->-=.6]        node [left]  {} (20)
         
    (11) edge [loop below]    node [right] {} (11)
         edge [->-=.6]        node         {} (21);
      \end{tikzpicture}
      \label{fig:udB3_3path}
      }
   }
   
   \caption{Unoriented de Bruijn graphs $uBg(k,n)$.  The graphs are undirected; the arrows on the edges denote the directions taken in the alternating Eulerian paths that generate  the following unoriented de Bruijn sequences of optimal length: (a) 0011, (b) 0011220, (c) 00010111, (d) 00010111212020122200.}
   \label{fig:udBpaths}
\end{figure}
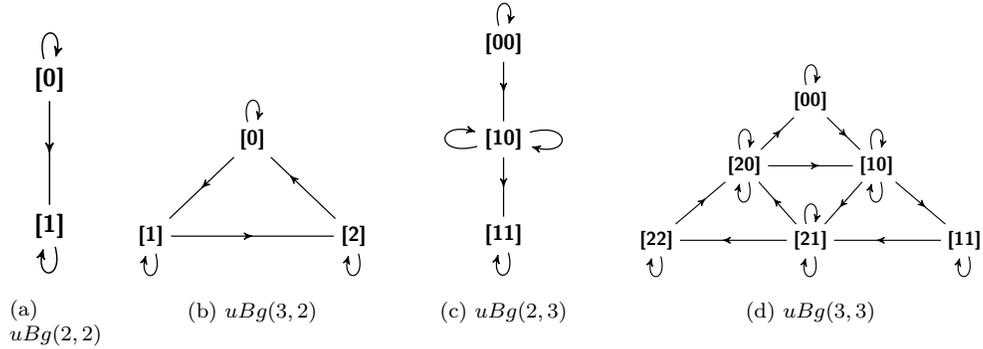  

%%%%%%%%%%%%%%%%%%
%%End figure
%%%%%%%%%%%%%%%%%%%

\section{Nontrivial unoriented de Bruijn sequences with optimal length}

\noindent For cases in which either $k$ is two or odd and $n \leq 3$, it is  possible to form an unoriented de Bruijn sequence with length $l(k,n)$ because a $uBg(k,n)$ graph will admit an alternating Eulerian path. The proof of this claim relies on the count of odd-degree vertices in the graph $uBg(k,n)$ which is determined by first considering the numbers of Type-I and Type-II incidences relative to each vertex, as given in Lemma~\ref{types}. Recall from Section~2 that additional undirected loops appear in unoriented de Bruijn graphs (those corresponding to edges $\bm{v} \rightarrow \bm{v'}$). Vertices with one additional loop have odd degree. 

The prefix and suffix  of a word $\bm{v}$ determine the edges associated to the vertex $[\bm{v}]$. More specifically, this particular substructure of a word determines the number of loops incident to it.  Every word ${\bm v}$ falls into one of three disjoint categories:

\begin{enumerate}
\item The prefix and suffix of $\bm{v}$ form non-palindromes, and there is no loop at $[\bm{v}]$.  For example, the prefix 100 and suffix 001 of $1001$ are not palindromes, and there is no loop at $[1001]$ in $uBg(2,5)$.
  
\item The prefix of exactly one of either $\bm{v}$ or $\bm{v'}$ form a palindrome, and there is exactly one loop at $[\bm{v}]$.  For example, 101 appears at the end of 1101 and at the beginning of 1011, and there is exactly one loop  $ \hookrightarrow [1101]$.

\item The prefix and suffix of $\bm{v}$ form palindromes, and there are exactly two loops at $[\bm{v}]$.  For example, there are two loops incident with $[1010]$ because 010 appears at the end of 1010 and at the beginning of 0101, and 101 appears at the end of 0101 and at the beginning of 1010.
\end{enumerate}   

In general, Lemma~\ref{types} shows that there is a 1-1 correspondence between the appearance of a palindrome as the prefix or suffix of a word and the existence of a loop $\hookrightarrow \bm{[v]}$ in $uBg(k,n)$ as exempified in the above enumeration. 

\begin{lemma}
\label{types}
Consider a vertex $[\bm{v}]$ in $uBg(k,n)$.  
If the prefix (suffix) of $\bm{v}$ is a non-palindrome, then the number of incidences of Type II (I) relative to $[\bm{v}]$ is $k$. 
If the prefix (suffix) of $\bm{v}$ is a palindrome, then the number of incidences of Type II (I) relative to $[\bm{v}]$ is $k+1$.  
\end{lemma}

\begin{proof}
For  $\bm{v} = v_1 v_2 ... v_{n-2} v_{n-1}$, there is an edge of Type I relative to $[\bm{v}]$ to each of the  vertices representing the equivalence classes of the words $\bm{w} = v_2 v_3....v_{n-1} u$, where there are $k$ choices for $u$.  Thus there are at least $k$ incidences of Type I relative to $[\bm{v}]$.  When the suffix of $\bm{v}$ is a non-palindrome, this is exactly the count. But, if the suffix of $\bm{v}$ is a palindrome, then ${\bm w} = v_{n-1} ...v_3 v_2 u$, which equals $\bm{v}'$ when $u = v_1$.  In this case, the edge between $[\bm{v}]$ and $[\bm{w}]$ is a loop $[\bm{v}] \hookrightarrow [\bm{v}]$, so there are $k+1$ incidences of Type I relative to $[\bm{v}]$. An analogous argument holds for edges of Type II relative to $[\bm{v}]$ and the prefix of $\bm{v}$.

\qed
\end{proof}

Note that if $\bm{v}$ is itself a palindrome, then each of its edges are of Type I and of Type II.  Rephrasing the result of Lemma~\ref{types} as in Corollary~\ref{typesCount} and Table~1 more directly addresses the conditions for the existence of an alternating Eulerian path: 

\begin{corollary}
\label{typesCount}
If $\bm{[v]}$ is a vertex in $uBg(k,n)$ representing a non-palindrome reflected pair, then $\bm{[v]}$ has an equal number of Type-I and Type-II incidences if deg\textnormal{(}$\bm{[v]}$\textnormal{)} is even, and the number of Type-I and Type-II incidences differ by one if deg\textnormal{(}$\bm{[v]}$\textnormal{)} is odd.
\end{corollary}

\begin{table}[H]
   \begin{center}
   \begin{tabular}{| l | l || l | l |}
\hline
prefix & suffix & Type I & Type II\\
\hline
\hline
p & p & $k+1$ & $k+1$ \\
\hline
 p (np) & np (p) & $k+1$ ($k$) & $k$ ($k+1$)\\
\hline
np & np & $k$ & $k$\\
\hline
   \end{tabular}
    \\[.75cm]
   \caption{Characterization of incidences of $\bm{[v]}$. The number of Type-I (II) incidences relative to a vertex $\bm{[v]}$ as determined by the prefix and suffix of $\bm{v}$. `p' denotes palindrome; `np' denotes non-palindrome.}
   \end{center}
     \label{table:lemmaChar}
\end{table}

Corollary \ref{typesCount} guarantees that every unoriented de Bruijn graph satisfies Conditions 2 and 3 of Proposition~\ref{alt_conditions}.  The only potential impediment to the existence of an alternating Eulerian path is therefore Condition 1 of Proposition~\ref{alt_conditions}, namely that there be exactly 0 or 2 vertices of odd degree. The proof of Theorem \ref{flip} proceeds by counting the number of vertices of odd degree, which turns out to be larger than 2 if either $k$ is even and larger than 2 or $n$ is larger than 3 and neither $k$ nor $n$ is 1.  

%%%%%%%%%%Theorem%%%%%%%%%%%%%%%
\begin{theorem}
\label{flip}
Nontrival unoriented de Bruijn sequences $uB(k,n)$ of optimal length exist if and only if $k$ is either two or odd and $n \leq 3$.   
\end{theorem}

\begin{proof}
Let $k$ and $n$ be integers larger than 1.  By Corollary \ref{typesCount}, the vertices $[\bm{v}]$ with $\bm{v}$ a non-palindome of odd degree are those for which exactly one of the prefix or suffix of $\bm{v}$ is a palindrome. The count of $[\bm{v}]$ with this property proceeds by first counting the number of length-$(n-1)$ words that contain a length-$(n-2)$ palindrome, which is $k(k^{(n-2)/2})$ when $n$ is even, and $k^2(k^{(n-3)/2})$ when $n$ is odd. We then subtract from these counts the number of length$-(n-1)$ words where both the first and the last length-$(n-2)$ subwords form a palindrome, of which there are $k$ when $n$ is even and $k^2$ when $n$ is odd (see Fig.~\ref{fig:p_p}). 

The count of odd-degree vertices $[\bm{v}]$ with $\bm{v}$ a palindome depends on the parity of $k$.  The degree of any constant vertex (that is, a vertex $[\bm{v}]$ for which all of the characters of $\bm{v}$ are the same) is $k+1$. Thus when $k$ is even, the count of constant vertices that have odd degree is $k$. When $k$ is odd, constant vertices have even degree, and as such, do not contribute to the count. Similarly, by Lemma~\ref{undir}, the degree of any non-constant palindrome vertex is $k$. The count of such vertices is $(k^{\floor{n/2}}-k)$. Therefore, the total count of the number $ov(k,n)$ of odd-degree vertices in $uBg(k,n)$ is

\begin{center}
%\# odd degree vertices 
$ov(k,n) =$ 
$\begin{cases}
k^{n/2}, & n\mbox{ even, }k\mbox{ even}\\
k^{(n+1)/2}-k^2+k, & n\mbox{ odd, }k\mbox{ even}\\
2(k^{n/2}-k), & n\mbox{ even, }k\mbox{ odd}\\
k^{(n+1)/2}+k^{(n-1)/2}-k^2-k, & n\mbox{ odd, }k\mbox{ odd}.\\
\end{cases}$
\end{center}

Observe that if $k$ is two or odd and $n\leq 3$ the count of odd-degree vertices remains less than or equal to 2. However, if either $k>2$ is even or $n>3$, then the count is greater than 2. Therefore, an Eulerian path  exists in $uBg(k,n)$ if $k$ is two or odd and $n \leq 3$. 

\qed
\end{proof}

\iffalse

Theorem~\ref{flip} implies that when $k>2$ is even or $n>3$, there will necessarily be more than two odd-degree vertices in $uBg(k,n)$.  Some edges must be traversed more than once in any path which visits all edges of $uBg(k,n)$. Thus, when either \textbf{$k$ is even or $n$  is greater than $3$}, no unoriented de Bruijn sequence can exist with the optimal length $l(k,n) = (k^n+k^{\lceil n/2 \rceil} + 2n - 2)/2$; some repetition of subwords is required to see a member of each reflected pair.

\fi

%%%%%%%%%%%%     pal_pal figures     %%%%%%%%%%%%%%%%%%%%%%

\begin{figure}[H]
 \centering
 \subfloat[]{
    \includegraphics[scale=.33]{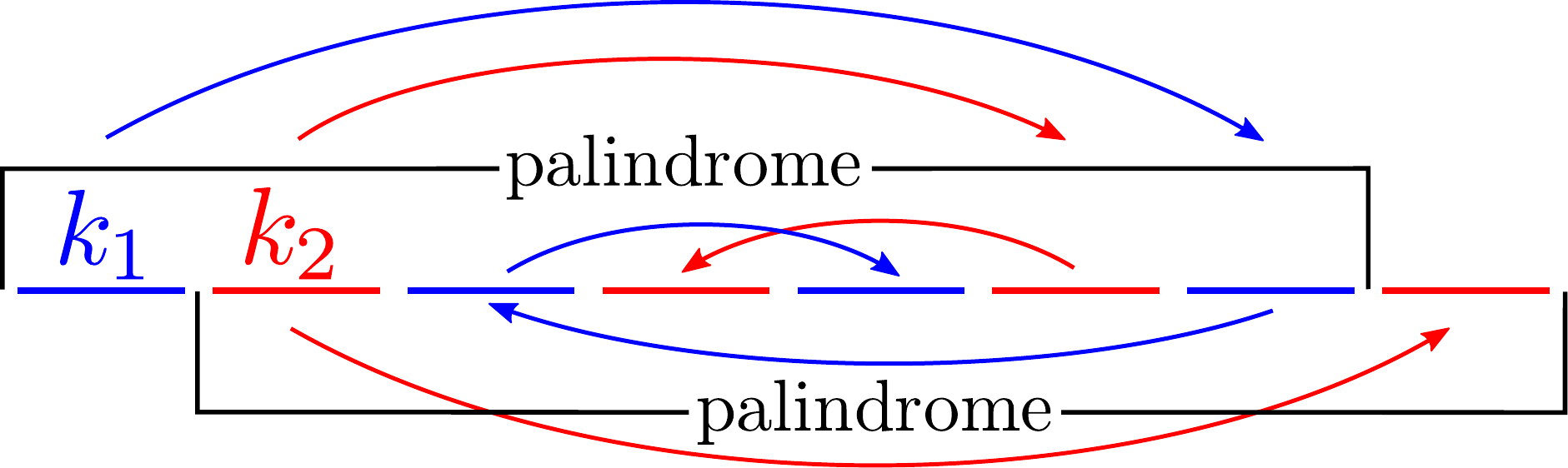}
    \label{fig:pal_pal_odd}
  }
  \subfloat[]{
    \includegraphics[scale=.33]{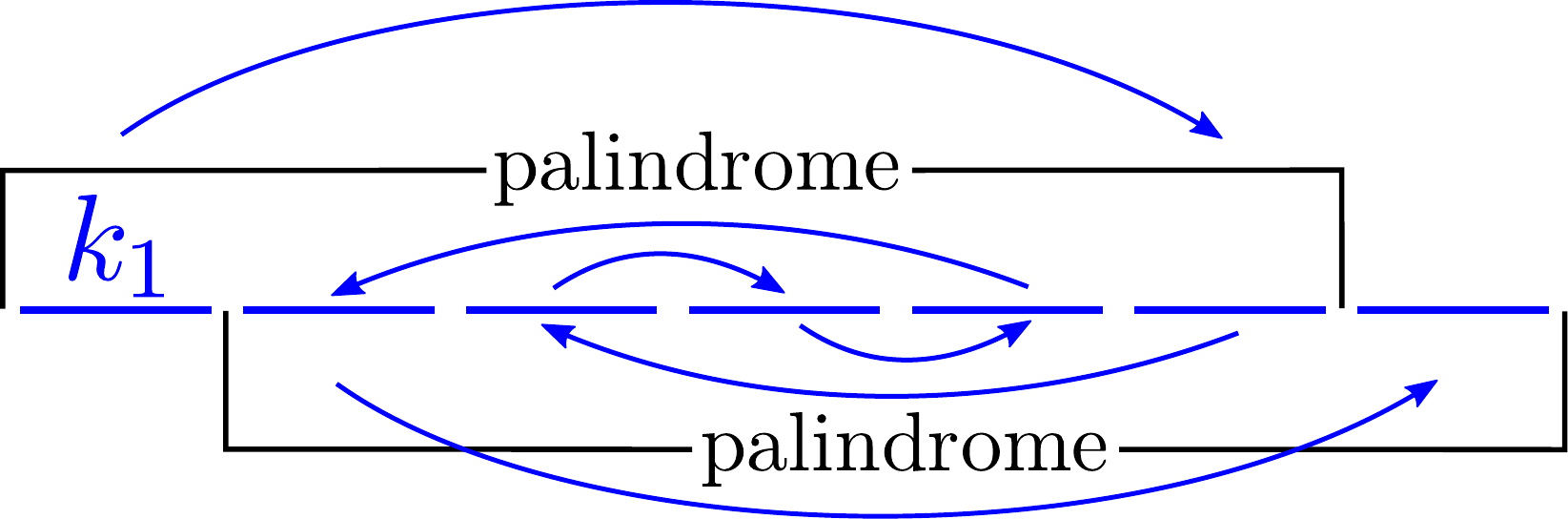}
    \label{fig:p_p_even}
    }
    \caption{(color online) Derivation of the count of words of length $n-1$ such that both the first and last $n-2$ characters form a palindrome when (a) $n$ is even and (b) $n$ is odd. The letters labeled $k_j$ represent a choice of a character from an alphabet $\Sigma_k$ at that position. The paths following the arrows (beginning with the choice of $k_j$) show which subsequent characters are determined by the palindromic conditions.}
    \label{fig:p_p}
\end{figure}
%%%%%%%%%%%%     end figure     %%%%%%%%%%%%%%%%%%%%%%%%

%%%%%%%%%%%%%%%%%%%%%%%%%%%
\section{Unoriented de Bruijn sequences of suboptimal length}
%%%%%%%%%%%%%%%%%%%%%%%%%%%

\noindent It has been established that nontrivial unoriented de Bruijn sequences of optimal length can be constructed from alternating Eulerian paths or circuits in unoriented de Bruijn graphs if and only if $n$ is not greater than 3 and $k$ is two or odd. The focus of this section is the construction of nontrivial unoriented de Bruijn sequences (of suboptimal length) for $k$ larger than two and even or $n$ larger than 3. Such a sequence will have at least one word from every reflected pair and have as few duplicate representations from reflected pairs as possible. That is, we are asking for the minimum number of edges which must be repeated in $uBg(k,n)$ so that the resulting graph has an alternating Eulerian path.  In general, a non-Eulerian graph can be \textit{Eulerized} by pairing odd-degree vertices and connecting pairs with paths formed by duplicating edges in the graph.  Any optimal Eulerization of $uBg(k,n)$ will retain two vertices of odd degree and will add as few duplicated edges as possible.

In Fig.~\ref{fig:udB}, we have redrawn undirected de Bruijn graphs with edge types indicated by edge end markers. A triangle at the end of an edge represents that the edge is of Type-I relative to the corresponding vertex, Type-II is unmarked, and dashed lines indicating duplicated edges in an Eulerization of the graph.  Consider first Fig.~\ref{fig:color_udB2_4}, showing $\overline{uBg}(2,4)$, having added an additional edge to ${uBg}(2,4)$. ${uBg}(2,4)$ has 4 odd-degree vertices, $[000],[111],[100],[110]$, and thus does not support an Eulerian path.  In order to Eulerize the graph, it is sufficient to add a Type-I edge between [100] and [110], since both vertices require a Type-I edge in order to be traversable twice. With this additional edge, there are only two vertices of odd degree. Additionally, each vertex that does not represent a palindrome has an equal number of Type-I and Type-II edges. Therefore, there is an alternating Eulerian path in the Eulerized graph. An example of a resulting unoriented de Bruijn sequence $uB(2,4)$ is $00001100101111$ which has both members of the reflected pair $[1100]$ when read forwards but no other redundancy.  The remaining subfigures of Fig.~\ref{fig:udB} illustrate other graphs and Eulerizations that admit alternating Eulerian paths.

Finding an optimal Eulerization of $uBg(k,n)$ can be seen as a variant of the well-known route inspection problem in which it is asked, ``for an undirected graph $G$, what is the minimum length path that visits every edge at least once?'' Solutions to the route inspection problem and many variants are known~\cite{edmonds}.  Our problem is a variation of the classic, undirected problem because duplicate edges added to $uBg(k,n)$ must be of the correct type.  Rather than modifying existing solutions to suit the peculiarities of unoriented de Bruijn graphs, we provide an upper bound for the number of duplicate edges needed to Eulerize $uBg(k,n)$ and thereby derive an upper bound for the length of unoriented de Bruijn sequences. 

Recall that the diameter of a graph is the maximum distance between any two vertices, where distance is the length of the shortest path between them.  The diameter of a (directed) de Bruijn graph, $Bg(k,n)$, can be seen to be $n-1$ by observing that any word of length $n-1$ can be transformed into any other word in $n-1$ shifts or less.  \textit{A priori} the diameter of $uBg(k,n)$ could be less than $Bg(k,n)$.  However, the observation that a valid sequence will be generated by a path in $uBg(k,n)$ only if the path enters and leaves each vertex on edges of different types is exactly the requirement that the path obeys the direction of edges in $Bg(k,n)$. This observation leads to the following proposition:

\begin{proposition}
	The length of an unoriented de Bruijn sequence, $uB(k,n)$, is bounded above by 
	$$
		l(k,n) + (n-1)[ov(k,n)/2 - 1],
	$$
	if $ov(k,n) > 2$.
	\label{prop:upperbound}
\end{proposition}

\begin{proof} The maximum distance between any two vertices in $uBg(k,n)$ is $n-1$ if we take distance to be measured by the shortest alternating path between them.  Thus, an upper bound on the alternating-path distance between any two odd-degree vertices in $uBg(k,n)$ is $(n-1)$.  Say there are $ov(k,n) > 2$ vertices of odd degree.  Then the number of duplicate edges in an optimal Eulerization of $uBg(k,n)$ is bounded above by $(n-1)[ov(k,n)/2 - 1]$.  Adding the number of repeated edges to $l(k,n)$ gives the bound on the length of $uB(k,n)$. \qed
\end{proof}

Proposition~\ref{prop:upperbound} insists that $ov(k,n) > 2$ because if $ov(k,n) = 2$ (or $0$) then $uBg(k,n)$ admits an alternating Eulerian path (or circuit) and so the length of $uB(k,n)$ is optimal and no duplicate edges are needed.  

Note that the bound given in Proposition~\ref{prop:upperbound} is achieved for some $k$ and $n$ when $ov(k,n) > 2$. In the case of $uBg(4,3)$ -- shown in Fig.~\ref{fig:color_udB4_3} -- the bound $(n-1)[ov(k,n)/2 - 1] = 2$ is the number of edges in an optimal Eulerization. Similarly, as shown in  Fig.~\ref{fig:color_udB2_4}, the graph $uBg(2,4)$ requires $(2-1)(4/2-1) = 1$ additional edge. On the other hand, the bound is not always strict, as exemplified by $uBg(2,5)$ (Fig.~\ref{fig:color_udB2_5}), where only 2 additional edges are needed to Eulerize the graph, while $(5-1)[ov(2,5)/2 - 1] = 8$.

To assess how far from optimal these unoriented de Bruijn sequences can be (in the worst case) we compute the ratio
$$
	r(k,n) = \begin{cases}  0,  & \text{ if } ov(k,n) = 0 \; \text{or} \; 2 \\  (n-1)[ov(k,n)/2 - 1]/l(k,n), & \text{ if } ov(k,n) > 2,  \end{cases}
$$
of the upper bound on the number of duplicate edges to the optimal length of an unoriented de Bruijn sequence.  Simple calculations show that in the limit as either $k$ or $n$ is taken to infinity, $r(k,n)$ converges to zero. Values of $r(k,n)$ are shown for small choices of $k, n \geq 2$ in Fig.~\ref{fig:ratio} as an illustration. Thus $k$ or $n$ can be chosen large enough so that the fraction of redundant reflected pairs in a suboptimal unoriented de Bruijn sequence is an arbitrarily small fraction of all reflected pairs.

\begin{figure}[H]
\captionsetup[subfigure]{width=1.9cm}
\centering 
\includegraphics[width=.5\textwidth]{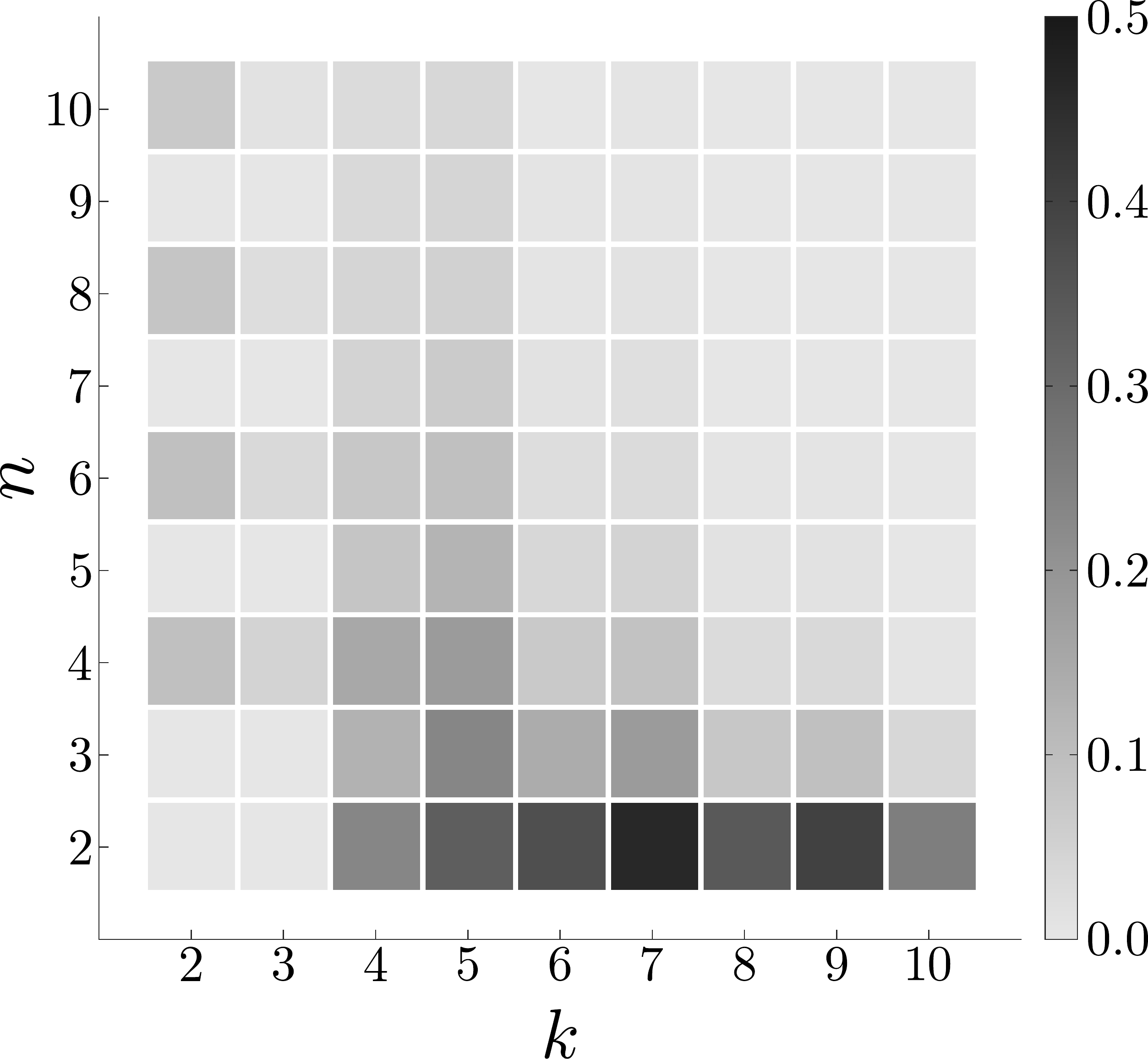} 
\caption{Plot of $r(k,n)$, the ratio of the upper bound on the number of duplicate edges in an optimal Eulerization of $uBg(k,n)$ to the optimal length $l(k,n)$ over the range $2 \leq n,k \leq 10$.}
\label{fig:ratio}
\end{figure}

\newpage

%%%%%%%%%%
%End Colored small graphs
%%%%%%%%%%%%%%%%%%%

%% For one-column wide figures use
%\begin{figure}
%% Use the relevant command to insert your figure file.
%% For example, with the graphicx package use
%  \includegraphics{example.eps}
%\end{figure}
%%
%% For two-column wide figures use
%\begin{figure*}
%% Use the relevant command to insert your figure file.
%% For example, with the graphicx package use
%  \includegraphics[width=0.75\textwidth]{example.eps}
%\end{figure*}
%\begin{table}
%\begin{tabular}{lll}
%\hline\noalign{\smallskip}
%\end{tabular}
%\end{table}

%\begin{acknowledgements}
%\end{acknowledgements}



\end{document}